\documentclass[11pt, a4paper]{article}

\usepackage[english]{babel}
\usepackage{hyphenat}

\usepackage{amsthm}
\usepackage{enumerate}
\usepackage{euscript} 
\usepackage{eufrak} 
\usepackage{mathrsfs}
\usepackage{amssymb}
\usepackage{stackrel}
\usepackage{mathtools}
\usepackage{ bbold }
\usepackage{mathpazo}
\usepackage[all]{xy}
\usepackage{hyperref}

\usepackage{geometry}
\usepackage{float}
\usepackage{cite}

\theoremstyle{plain}

\newtheorem{theorem}{Theorem}[section]
\newtheorem{proposition}[theorem]{Proposition}%[section]
\newtheorem{cor}[theorem]{Corollary}%[section]
\newtheorem{lemma}[theorem]{Lemma}%[section]

\theoremstyle{definition}
\newtheorem{obs}[theorem]{Remark}%[section]

\newtheorem{defi}[theorem]{Definition}%[section]
\newtheorem{ej}[theorem]{Example}

\newcommand{\ent}{\mathbb Z}
\newcommand{\re}{\mathbb R}

\newcommand{\oct}{\mathbb{O}}
\newcommand{\C}{\mathcal{C}}
\newcommand{\U}{\mathcal{U}}
\newcommand{\M}{\mathscr{M}}

\newcommand{\x}{\otimes}
\newcommand{\X}[1]{\otimes_{#1}}

\renewcommand{\1}{\mathbb{1}}
\newcommand{\epi}{\twoheadrightarrow}
\newcommand{\mono}{\hookrightarrow}
\newcommand{\tens}[1]{\mathbin{\mathop{\otimes}\limits_{#1}}}
\newcommand{\Proj}[1]{\mathbb P^{n}_{#1}}
\renewcommand{\mod}{\mathcal{M}od}%para categoria de modulos
\newcommand{\Hom}[3]{\mathrm{Hom}_{#3} (#1, #2)}

\newcommand{\cenesp}{\vcenter{\xymatrix{\quad \quad \quad}}}

\author{Matias Data\footnote{The author was fully supported by CONICET, Argentina and the Universidad de Buenos Aires.},~
Juliana Osorio \footnote{The author was fully supported by CONICET, Argentina and the Universidad de Buenos Aires.}}
\title{On the Relative Projective Space}
\date{}
\begin{document}
\maketitle
\begin{abstract}
Let $(\C,\x,\1)$ be an abelian symmetric monoidal category satisfying certain exactness conditions. In this paper we define a presheaf  $\Proj{\C}$ on the category of commutative algebras in $\C$ and we prove that this functor is a $\C$-scheme in the sense of Toen and Vaqui\'e. This construction gives us  a context of non-associative relative algebraic geometry. The most important example of the construction is the octonionic projective space.
\end{abstract}

\tableofcontents

\section*{Introduction}
The study of the octonionic projective plane was initiated by R. Moufang in 1933 \cite{moufang}. She constructed it by coordinatizing using the octonion algebra, also known as the Cayley-Dickson algebra. Her point was to show an example of a non-Desarguessian plane.\\
Another way to approach the octonionic plane is via Jordan algebras. The idea is to consider the exceptional simple Jordan algebra $H(\oct_{3})$ of $3\times 3$ matrices with entries in the octonions, which are symmetric with respect to the involution. This attempt was first made by P. Jordan in 1949 \cite{jordan}. He considered the real octonion algebra and used the idempotents of $H(\oct_{3})$ to represent the points and lines in the octonionic projective plane. Later in 1953, H. Freudenthal rediscovered the same construction \cite{freudenthal} and used it to study the exceptional Lie groups $F_{4}$ and $E_{6}$. In this direction other attempts were made, but allowing the octonion algebra over a  field of characteristic not 2 or 3. In this setting, the elements of rank one were used to represent points and lines.\\

The main purpose of this work is to give a new construction for the octonionic projective space, which we shall denote $\Proj{\oct}$. The approach is via relative algebraic geometry in monoidal categories. The relative algebraic geometry over a symmetric monoidal category has been widely studied in the literature, see for instance \cite{DeligneMilne}, \cite{Saav}, \cite{hakim} and \cite{TV}. When the monoidal category is the category of modules over a commutative ring $k$, the relative algebraic geometry reduces to the usual algebraic geometry over the scheme $Spec(k)$. \\
More explicitely, let $(\C,\x,\1)$ be a closed monoidal category with limits and colimits and $Comm(\C)$ the category of commutative algebras in $\C$, then the category of affine $\C$-schemes $Aff_{\C}$ is defined as $Comm(\C)^{op}$. Next,  a $\C$-scheme will be a sheaf in $Aff_{\C}$ which is covered by finitely many affine schemes 

Throughout this paper $(\C,\x,\1)$ is an abelian bicomplete symmetric closed monoidal category such that $\1$ is a projective finitely presentable generator. This condition on $\1$ means that  the forgetful functor 
$V_{0}=Hom_{\C}(\1,-):\C\to \mathcal Ens$
 is conservative, preserves and reflects epimorphisms and filtered colimits. Although not all of these properties are needed in  some of the results, these are exactly the conditions required for the functor $\Proj{\C}$ to be a $\C$-scheme. A category fulfilling those conditions will be called an {\it abelian strong relative context}. Because of the adjunction $\C\to\mod_{\C}(A)$, if $\C$ is an abelian strong relative context then $\mod_{\C}(A)$ is also an abelian strong relative context.\\
An outline of this work is the following: the first section contains a short review of the ideas of relative algebraic geometry developed in \cite{TV}.  Section 2 deals with Zariski coverings of an affine scheme $Spec(A)$ in terms of a generating family of elements in $A$.  We show that associated to an ideal of $A$ there is an open sub-scheme of $Spec(A)$ and we also give a sufficient condition for a family of Zariski open immersions to be a Zariski covering. This last result will allow as to show that the functor $\Proj{\C}$ has a finite Zariski covering by affine schemes. Finally we show that for a faithfully flat morphism $A\to B$ in $Comm(\C)$ and $M\in \mod_{\C}(A)$, $L\mono M$ is a direct summand whenever $B\x L$ is a direct summand of $B\x M$. In section 3 we define the functor $\Proj{\C}$ and we prove that if $\C$ is an abelian strong relative context then $\Proj{\C}$ is a $\C$-scheme. Following the ideas of H. Albuquerque and S. Majid \cite{quasialgebraO} we are able to define the category of $\oct$-modules, $\mod(\oct)$ and we prove that this category is an abelian strong relative context, as a consequence we have defined the relative scheme $\Proj{\oct}$. 
\section{Relative Algebraic Geometry}

%\subsection*{Zariski Topology in $Aff_{\C}$}
Let $T$ be any category  with finite limits and consider the pseudo functor $M:T^{op}\to Cat$ such that
\begin{enumerate}
\item For every $X$ in $T$, the category $M(X)$ posses all limits and colimits.
\item For every $f:X'\to X $ in $T$, the functor $M(f)=f^{*}:M(X)\to M(X')$ has a conservative right adjoint $f_{*}:M(X')\to M(X)$.
\item For every cartesian diagram in $T$
$$
\xymatrix{Y'\ar^{f'}[r]\ar_{g'}[d]& Y\ar^{g}[d]\\
X'\ar_{f}[r]& X}
$$
the natural transformation $f^{*}g_{*}\Rightarrow g'_{*}f'^{*}$ is an isomorphism.
\end{enumerate}

\begin{defi} Let $(f_{i}:X_{i}\to X)_{i\in I}$ be a family a morphism in $T$.
\begin{enumerate}
\item The family is an $M$-cover if there exists a finite subset $J\subset I$ such that  the family of the induced functors $(f^{*}_{i}:M(X)\to M(X_{i}))_{i\in J}$ is jointly conservative.
\item The family is said to be $M$-flat if the functors $f_{i}^{*}$ are left exact for every $i\in I$.
\item The family is $M$-faithfully flat if it is an $M$-cover and $M$-flat. 
\end{enumerate}
\end{defi}

\begin{proposition}
The $M$-faithfully flat families define a pretoplogy over $T$. 
\end{proposition}
\begin{theorem}\label{M stack}
The seudo-functor $M$ is a stack.
\end{theorem}
\begin{obs}
Given $\mathcal U=(U_{i}\to X)_{i\in I}$ an $M$-covering in the site $T$, one has the category $Desc(\mathcal U/X,M)$ of descent data, this is a category whose objects are pairs $(x_{i},\theta_{i,j})_{i,j}$ with $x_{i}$ an object in $M(U_{i})$
and $\theta_{i,j}:(x_{i})|_{U_{i,j}}\cong (x_{j})|_{U_{i,j}}$ are isomorphisms in $M(U_{ij})$ satisfying the cocycle condition $\theta_{j,k}\circ\theta_{ij}=\theta_{ik}$ in $M(U_{ijk})$, where $U_{ij}$ denotes the pullback $U_{i}\times_{X}U_{j}$. A morphism between two descent data $(x_{i},\theta_{ij})_{ij}, (y_{i},\phi_{ij})_{ij}$ is a family of morphisms $f_{i}:x_{i}\to y_{i}$ in $M(U_{i})$ compatible with the given isomorphisms, i.e., $\phi_{i,j}f_{i}=f_{j}\theta_{ij}$ in $M(U_{ij})$. The fact the $M$ is a stack can be paraphrased as follows: for each covering $\mathcal V$ there is canonical functor $p^{*}:M(X)\to Desc(\mathcal V/X,M)$. $p^{*}$ is in fact an adjoint equivalence with right adjoint given by
\begin{equation}\label{M stack}
p_{*}(x_{i},\theta_{i,j})=Lim\left(\xymatrix{\displaystyle\prod_{i}(p_{*})(x_{i})\ar@<0.5ex>[r]\ar@<-0.5ex>[r]& \displaystyle\prod_{i,j}(p_{i,j})_{*}(x_{i})|_{U_{i,j}}}\right)
\end{equation}
with $p_{i,j}:U_{i,j}\to X$.
\end{obs}

If $(\C,\x,\1)$ a symmetric monoidal category satisfying the conditions stated before and $Comm(\C)$ denotes the category of commutative algebras in $\C$, then the category of affine schemes over $\C$ is defined as $Aff_{\C}:=Comm(\C)^{op}$, the seudo functor $M$ assigns to each affine scheme $X=Spec(A)$, the category of $A$-modules $\mod_{\C}(A)$ and for any morphism $f:Spec(B)\to Spec(A)$, $f^{*}:\mod_{\C}(A)\to \mod_{\C}(B)$ is given by the base change $-\X{A}B$. The topology induced by the seudo functor $M$ is called the faithfully flat quasi-compact  (fpqc) topology. The Zariski topology in $Aff_{\C}$ is defined as follows:
\begin{defi}
The family $(f_{i}:X_{i}\to X)_{i\in I}$ in $Aff_{\C}$ is a Zariski covering if it is an $M$-faithfully flat family such that each morphism $f_{i}:X_{i}\to X$ is an epimorphism of finite presentation.
\end{defi}

\begin{cor}\label{subcanonical}
For every $X\in Aff_{\C}$ the presheaf $h_{X}$ is a sheaf with respect to the faithfully flat topology.
\end{cor}
As in the classical setting in algebraic geometry, a relative scheme is that of a sheaf which has a Zariski open covering by affine schemes. In order to define $\C$-schemes the Zariski topology has to be extended to $Sh(Aff_{\C})$.

\begin{defi}\hfill
\begin{enumerate}
\item Let $X\in Aff_{\C}$ and $F\subset X$ a sub sheaf.  $F$ is said to be an open Zariski of $X$ if there exists a family of open Zariski $\{X_{i}\to X\}_{i\in I}$ in $Aff_{\C}$ such that  $F$ is the image of the morphisms of sheaves $\coprod_{i\in I}X_{i}\to X$.

\item $f:F\to G$ in $Sh(Aff_{\C})$ is an open Zariski (open Zariski immersion, open sub functor) if for every affine scheme $X$ and every morphism $X\to G$ the induced morphism $F\times_{G}X\to X$ is a monomorphism with image a Zariski open of $X$, i.e., $F\times_{G}X$ is a Zariski open of $X$.
\end{enumerate}
\end{defi}
\begin{defi}
A sheaf $F\in Sh(Aff_{\C})$ is a scheme relative to $\C$ or a $\C$-scheme if there exists a family $\{X_{i}\}_{i\in I}\in Aff_{\C}$ such that for all $i$ there exists $X_{i}\to F$ satisfying
\begin{enumerate}
\item The morphism $X_{i}\to F$ is a Zariski open of $F$ for all $i$.
\item The induced morphism $p:\coprod_{i\in I}X_{i}\to F$ is an epimorphism of sheaves.
\end{enumerate}
 \end{defi}
 
%if $\1$ is of finite presentation and projective in $\C$, $A$ is also of finite presentation and projective in $\mod_{\C}(A)$. 

\section{Some Commutative Algebra in symmetric categories.}\label{defs preliminares}
In this section we prove several lemmas needed in order to prove that what we define as the projective space is in fact a $\C$-scheme. These lemmas are the relative version of very well-known results in algebraic geometry.\\
% Lemmas \ref{partition of unity} and \ref{generating is a zariski cov} are usuful to prove Lemma \ref{complementary open}. This Lemma, together with Lemma \ref{epi of sheaves} play an important role, as they are needed to prove that the projective space has a Zariski covering. In this section we also define {\it line objects} and give some of their properties, as well as examples in very well-known categories like $A$-mod and $Qcoh(X)$.
For $A\in Comm(\C)$, we say that $(f_{i})_{i\in I}\subset Hom_{A}(A, A)$ is a generating family of $A$ if $\coprod_{i}\varphi_{f_{i}}:\coprod A\to A$ is an epimorphism. A finite collection of $(f_{i})_{i\in J}\subset Hom_{A}(A, A)$ is a partition of unity if there exists arrows $(s_{i})_{i\in J}\subset Hom_{A}(A,A)$ such that $\sum_{i\in J}s_{i}f_{i}=1$. 

\begin{lemma}[Partition of Unity]\label{partition of unity}
Let $(f_{i}:A\to A)$ be a generating family, then $(f_{i})_{i\in I}$ form a partition of unity on $A$.
\end{lemma}
\begin{proof}
Let us see that $(f_{i})_{i\in I}$ can be reduced to a finite family. In fact, for each finite subset $J=\{i_{1},\cdots i_{k}\}\subset I$ consider the generated ideal $I_{J}=<f_{i_{1}},\cdots,f_{i_{k}}>$. Then, these ideals determine a filtered diagram as shown above
$$
\xymatrix@R=0.8ex{
<f_{i}>\ar^{}[rd]	&	& & \\
	&	<f_{i},f_{j}>\ar^{}[rd]	& &\\
<f_{j}>\ar[ru]\ar[rd]	&		&  <f_{i},f_{j},f_{k}> & \cdots\\
	&	<f_{j},f_{k}>\ar[ru]	&  &\\\
<f_{k}>\ar[ru]	&	&	&	
}
$$
Since the family $(f_{i})_{i\in I}$ is epimorphic in $\mod_{\C}(A)$ then $A$ is the filtered colimit of these ideals, i.e., 
$$A\cong colim <f_{i_{1}},\cdots,f_{i_{k}}>.$$
Because $A$ is finitely presented in $\mod_{\C}(A)$,  we have the isomorphism
$$
Hom_{A}(A,A)\cong colim_{J\subset I}Hom_{A}(A,I_{J}).
$$
Then there exists and index $k$ such that the identity arrow $1:A\to A$ factorizes through $<f_{i_{1}},\dots,f_{i_{k}}>$, that is to say $A\cong<f_{i_{1}},\cdots,f_{i_{k}}>$.\\
Now, let us see that the finite family indexed by $J$ is a partition of unity. 
%the diagram above is commutative
%$$
%\xymatrix@C=0.8ex@R=0.7ex{
%A\ar^{1}[rr]\ar[rd] &	& A\\
%		& \quad<f_{i}>_{i\in J}\ar[ru]
%}
%$$ 
As we have an epimorphism $\xymatrix@-0.8pc{\coprod A\ar@{->>}^{~~(f_{i_{j}})}[r]& A}$ and $A$ is projective, there is a surjection 
$$
\xymatrix{Hom_{A}(A,\coprod_{i\in J}A)\ar@{->>}[r]^{~~(f_{i_{j}})^{*}} &Hom_{A}(A,A)}.
$$ 
Using the isomorphism $Hom_{A}(A,\coprod_{i\in J}A)\cong \prod_{i\in J}Hom_{A}(A,A)$, we have that for the identity arrow $1:A\to A$ there exists a family $(s_{i})_{i\in J}$ such that $\sum_{i\in J}s_{i}\circ f_{i}=1$
\end{proof}

\begin{lemma}\label{generating is a zariski cov}
Let $(\xymatrix@-1pc{A\ar^{f_{i}}[r]& A})_{i\in I}$ be a generating family. Then $(Spec(A_{f_{i}})\to Spec A)_{i}$ is a Zariski covering.
\end{lemma}

\begin{proof}
Each $A\to A_{f_{i}}$ is a flat epimorphism of finite presentation. By Lemma \ref{partition of unity}, there exists a finite subset $J\subset I$ such that the family $(f_{j})_{j\in J}$ is a partition of unity. In \cite[Proposition 2.7]{schemesmoncat} the author shows that the family of functors $\mod (A)\to \prod_{i} \mod(A_{f_{i}})$ is jointly conservative.
\end{proof}
\begin{defi}
Let $X=Spec A$ in $Aff_{\C}$, $I\mono A$ an ideal.  There is a subfunctor of $X$ associated to the ideal $I$ defined by: $U_{I}(B)=\{u:A\to B~:~BI\cong B\}$ where $BI=Im(\xymatrix{B\x I\ar^{u\circ j_{I}\x B}[r]& B\x B\ar^{m_{B}}[r] & B})$.
\end{defi}

\begin{lemma}[Complementary open sub-scheme]\label{complementary open}
 $U_{I}$ is  a $\C$-scheme.
\end{lemma}

\begin{proof}

First we prove that $U_{I}$ is a sub sheaf. Let $(B\to B_{i})_{i\in J}$ be a Zariski covering and let $(f_{i})_{i}$ be a compatible family in $\prod_{i}U(B_{i})\mono \prod_{i}h_{A}(B_{i})$. Since $h_{A}$ is a sheaf, there exists a unique $f\in h_{A}(B)$ whose restrictions to every open $Spec(B_{i})$ is $f_{i}$. Let us check that this $f$ is in fact a section in $U(B)$, i.e., $f:A\to B$ induces an isomorphism  $BI\cong B$. Since the $B_{i}$ form an open covering for $B$ we have that family  of functors  
$$
-\X{B}B_{i}:\mod (B)\to \mod(B_{i})
$$
is jointly conservative, so if we consider the inclusion $BI\mono B$, we know that for every $i\in J,~$ $BI\X{B}B_{i}\cong B_{i}I\overset{\sim}{\rightarrow}B_{i}$, therefore $BI\cong B$.

%We now prove that $U$ has a Zariski open covering:
Now we show that if $(f_{i})_{i}\subset Hom_{A}(A, A)$ is a generating family of the ideal $I$ then $U_{i}=Spec(A_{f_{i}})\to U$ is a Zariski open immersion and $\{U_{i}\to U\}_{i\in J}$ is a Zariski covering.  First, note that by the universal property of localizations 
\begin{equation}\label{U_{i}}
U_{i}(B)=Hom_{Comm(\C)}(A_{f_{i}},B)\cong\{f:A\to B~:~ B<f_{i}>\cong B\}.
\end{equation}
Moreover, the inclusion $U_{i}\to Spec(A)$ induces a morphism $U_{i}\to U$, by (\ref{U_{i}}) this morphism is a monomorphism. We will check that this morphism is in fact a Zariski open immersion. 
Let $Spec B\in Aff_{\C}$ and  $u:SpecB\to U$ and consider the pullback diagram
$$
\xymatrix@-1pc{
U_{i}\times_{A}SpecB\ar^{}[r]\ar^{}[d] & U_{i}\ar^{}[d]\ar@/^{1pc}/^{}[rdd] & \\
Spec B\ar^{u}[r]\ar@/_{1pc}/^{}[rrd]                      & U\ar^{}[rd]     &\\
                                                                        &                      & Spec A,
}
$$
we have to prove that $U_{i}\times_{A}SpecB\to SpecB$ is a Zariski open immersion. To give the morphism $u:SpecB\to U$ is the same as giving an element in $U(B)$, that is to say, a morphism $ u:A\to B$ such that $IB\cong B$, then the result follows by the isomorphism $U_{i}\times _{A} Spec B\cong Spec B_{u(f_{i})}$, where $u(f_{i}):\xymatrix{A\ar^{f_{i}}[r] &A\ar^{u}[r]& B}$ and $SpecB_{u(f_{i})}\to Spec B$ is a Zariski open, therefore $U_{i}\to U$ is a Zariski open.\\
On the other hand,  in view of $BI\cong B$,  $(u(f_{i}))_{i}$ is a generating family of B as an $A$-algebra. This family can be reduced to a finite family $(u(f_{j}))_{j\in J}$, thus by Lemma \ref{partition of unity}, $\coprod_{j\in J}Spec B_{uf_{j}}\to Spec B$ is an epimorphism of sheaves so is $\coprod_{j\in J}U_{j}\to U$
\end{proof}

We now give a sufficient condition for a morphism of sheaves to be an epimorphism. This result is analogous to its classical counterpart and it is very useful in order to prove that the projective space is in fact a scheme as it is covered by affine Zariski open immersions.
\begin{lemma}\label{epi of sheaves}
Let $\{U_{i}\to F\}$ be a finite family of affine Zariski open immersions in $Sh(Aff_{\C})$. If for every field object $K\in Comm(\C)$,  $~\coprod_{i}U_{i}(K)\to F(K)$ is surjective then $\coprod_{i} U_{i}\to F$ is an epimorphism of sheaves.
\end{lemma}
\begin{proof}
It is enough to prove the lemma for $F=SpecA$ since a necessary and sufficient condition for $G\to F$ to be a sheaf epimorphism is that for every affine scheme $Spec A$, $SpecA\times_{F}G\to Spec A$ is an epimorphism. In this case, we have to check that for each $\xymatrix{U_{j}=Spec(A_{j})\ar^{\quad u_{j}}[r] &Spec(A)}$, the family of functors $\mod(A)\to \mod(A_{j})$ is jointly conservative.\\
Let $0\neq M\in \mod(A)$, we will prove that $M_{j}:=A_{j}\X{A}M\neq 0$ for all j. As $M\neq 0$, then $M$ contains a submodule of the form $A/I$.  In fact, there is a non zero $f:A\to M$, so we take $I=ker f$, then we have the factorization
$$
\xymatrix@C=0.9ex{
A\ar^{f}[rr]\ar@{->>}[rd]  &  & M\\
	& A/I\quad~~\ar@{>->}[ru]	&
}
$$
 Let $\mathfrak m$ be a maximal ideal containing $I$, its existence is proven in Proposition \ref{existence of prime ideals}, then the morphism $\varphi$ from $A$ to the field object  $K=A/\mathfrak m$ represents an element in $F(K)$. As we have a surjective function $\coprod_{i}U_{i}(K)\to F(K)$, the element $\varphi$ seen as an arrow factorizes through some $u_{j}:A\to A_{j}$, this means that there exists $\varphi_{j}$ such that the diagram commutes
 $$
 \xymatrix@C=2.5ex{ A\ar_{u_{j}}[rd]\ar^{\varphi}[rr] &  &  K\\
 	& A_{j}\ar_{\varphi_{j}}[ru] & .
 }
 $$
Now, by the universal property of $Ker\varphi_{j}$, there exist a unique morphism $\mathfrak m\to Ker\varphi_{j}$, then we have the pullback diagram
$$
\xymatrix@-0.8pc{
\mathfrak m \ar@{-->}^{}[rd] \ar@/_1pc/[rdd]\ar@/^{1pc}/[rrd]  &    &       \\
            & u_{j}^{-1}(Ker \varphi_{j})\ar^{}[r] \ar@{>->}^{}[d] &  Ker \varphi_{j}\ar@{>->}^{}[d]\\
           & A\ar^{u}[r]\ar^{\varphi}[d]   & A_{j}\ar^{\varphi_{j}}[d]    \\
           & K\ar@{=}^{}  [r]  &    K
}
$$
with the morphism $\mathfrak m\to u_{j}^{-1}(Ker\varphi_{j})$ being a monomorphism.\\
Let $\mathfrak m_{j}$ be a proper maximal ideal containing  $Ker(\varphi_{j})$, since $u_{j}$ is flat we have that $u_{j}^{-1}(Ker\varphi_{j})\mono u^{-1}_{j}(\mathfrak m_{j})$, then $\mathfrak m\mono u^{-1}_{j}(\mathfrak m_{j})$. We claim that $u^{-1}(\mathfrak m_{j})$ is a proper ideal of $A$.  In fact, if $u^{-1}(\mathfrak m_{j})=A$, then the morphism $u_{j}:A\to A_{j}$ factorizes through $A\epi \mathfrak m_{j}$, but since $\mathfrak m_{j}$ is a proper ideal this is a contradiction. \\
By maximality $\mathfrak m=u^{-1}_{j}(\mathfrak m_{j})$. Then we have the commutative diagram 
$$
\xymatrix{
\mathfrak m_{}\ar^{}[r]\ar@{>->}[d]&  \mathfrak m_{j}\ar@{>->}[d]\\
A\ar^{u_{j}}[r]	& A_{j}
}
$$
tensoring with the $A$-algebra $A_{j}$ we have a morphism $A_{j}\X{A}\mathfrak m_{j}\cong A_{j}\mathfrak m\longrightarrow \mathfrak m_{j}$ commuting with the inclusion to $A_{j}$. Then this morphism must be a monomorphism. On the other hand, we have a monomorphism $\xymatrix{A_{j}I\quad\ar@{>->}[r]& A_{j}\mathfrak m}$, it follows that $\mathfrak m_{j}$ contains the ideal $A_{j}I$ then $A_{j}/A_{j}I\neq 0$ and we have a monomorphism
$$
\xymatrix{A/I\X{A}A_{j}\cong A_{j}/A_{j}I\quad\ar@{>->}^{} [r]&A_{j}\X{A}M=M_{j},}
$$
this means that $M_{j}\neq 0$ for all $j$, therefore $\{U_{i}\to Spec A\}_{i}$ is a Zariski covering.
 \end{proof}

%%%%%%%%%%%%Lema 3.15 demazure

The next tool we need in order to construct the projective space is the following two lemmas, they are the relative version of a well-known result in commutative algebra, it concerns about the stability of direct summands of a finitely presented module. We first introduce some notation. Let $A\to B$ be a morphism in $Comm(\C)$ and let $M,N$ be two $A$-modules, we would like to define a morphism 
$$
\zeta: B\X{A}hom_{A}(M,N)\to hom_{B}(B\X{A}M,B\X{A}N).
$$
We have the morphism $1\x\varepsilon:B\X{A}M\X{A}hom_{A}(M,N)\to B\X{A}N$ which by adjunction corresponds to a morphism 
$$
\xymatrix{hom_{A}(M,N)\ar^{\chi\qquad\quad}[r] &hom_{A}(B\X{A}M,B\X{A}N).}
$$
On the other hand, as $B\X{A}M$ and $B\X{A}N$ are $B$-modules, the object $hom_{A}(B\X{A}M,B\X{A}N)$ is also a $B$-module, with action 
$$
\mu:B\X{A}hom_{A}(B\X{A}M,B\X{A}N)\to hom_{A}(B\X{A}M,B\X{A}N)
$$
by composing these two morphism, we have the morphism
$$
\xymatrix{\xi:B\X{A}hom_{A}(M,N)\ar^{1\x \chi\qquad}[r] & B\X{A}hom_{A}(B\X{A}M,B\X{A}N)\ar^{\quad\mu}[r]& hom_{A}(B\X{A}M,B\X{A}N).
}
$$ 
It's not hard to see that $\xi$ equalizes the two morphisms 
$$
\xymatrix{hom_{A}(B\X{A}M,B\X{A}N)\ar@<-0.5ex>[r]\ar@<0.5ex>[r]& hom_{B}(B,hom_{A}(B\X{A}M,B\X{A}N))
}
$$ 
and since $hom_{B}(B\X{A}M,B\X{A}N)$ is, by definition the equalizer of these two morphisms, there exists an arrow $\zeta:B\X{A}hom_{A}(M,N)\to hom_{B}(B\X{A}M,B\X{A}N)$.\\

With notations as above we have the following results:
\begin{lemma}\label{lema 4.6}
If $A\to B\in Comm(\C)$ is faithfully flat and $M$ is finitely presentable, then the induced morphism 
$$
\zeta:B\tens{A}hom_{A}(M,N)\to hom_{B}(B\tens{A}M,B\tens{A}N)
$$
is an isomorphism.
\end{lemma}
\begin{proof}
The proof of this result is verbatim of the classical result given in \cite[Proposition 2.10]{CommAlgEisenbud}. The key is that under our hypothesis on $\C$, an object $M\in \C$ is finitely presentable in the sense that the functor $Hom_{\C}(M,-)$ preserves filtered colimits if and only if it has a finite presentation, that is to say, there exist integers $n,m$ such that the following diagram is exact
$$
\xymatrix{\1^{m}\ar^{}[r]& \1^{n}\ar^{}[r]& M}
$$
%That $M$ is finitely presentable if and only if $M$ has a finite presentation.\\
%Suppose $M=A$:
%$$
%\xymatrix{B\X{A}hom_{A}(A,N)\ar^{\cong}[d]\ar^{\zeta\quad}[r] & hom_{B}(B\X{A}A,B\X{A}N)\ar_{\cong}[d]\\
% B\X{A}N\ar_{\cong\quad\qquad\qquad}[r] & hom_{B}(B,B\X{A}N)\cong B\X{A}N
% }
%$$
%we have that $\zeta$ is an isomorphism.\\
%Now suppose $M=A^{m}$. Since $hom_{A}(M,-), B\X{A}-$ are additive functors, we have the isomorphisms
%$$
%\xymatrix{
%B\X{A}hom_{A}(A^{m},N)\ar^{\cong}[d]\ar^{\zeta\quad\quad}[r] & hom_{B}(B\X{A}A^{m},B\X{A}N)\ar_{\cong}[d]\\
% B\X{A}N^{m}\ar_{\cong\quad\quad\quad\quad\quad}[r]	& hom_{B}(B^{m},B\X{A}N) \cong B\X{A}N^{m}
% }
%$$
%with $\zeta$ the direct sum of the isomorphisms given in the previous case.\\%$\oplus_{1}^{m}\varphi_{A}$.\\
%For the general case, let $A^{m}\to A^{n}\to M\to 0$ be a finite presentation for $M$. Tensoring this presentation with $B$ we obtain again an exact sequence $$B^{m}\to B^{n}\to B\X{A}M\to 0.$$
%Applying $B\X{A}hom_{A}(-,N)$ and $hom_{B}(-,B\X{A}N)$ to both exact sequences respectively, we obtain the commutative diagram with exact rows
%$$
%\xymatrix{
%0\ar^{}[r] & B\tens{A}hom_{A}(M,N)\ar^{}[r]\ar@{-->}[d] & B\tens{A}hom_{A}(A^{n},N)\ar^{}[r]\ar^{\cong}[d] & B\tens{A}hom_{A}(A^{m},N)\ar^{\cong}[d]\\
%0\ar^{}[r] & hom_{B}(B\tens{A}M,N)\ar^{}[r] & hom_{B}(B\tens{A}A^{n},N)\ar^{}[r] & hom_{B}(B\tens{A}A^{m},N),
%}
%$$
%The exactness in the left of the first row is due to the fact $A\to B$ is faithfully flat.  The first column is an isomorphism since they are both the limits of isomorphic sequences.
\end{proof}

\begin{lemma}\label{lema 3.15 Demazure}
Let $A\to B\in Comm(\C)$ faithfully flat, $M\in \mod_{}(A)$ of finite presentation and $L$ an $A$-submodule of $M$. If $B\X{A}L$ is a direct summand of $B\X{A}M$ then $L$ is a direct summand of $M$.
\end{lemma}
\begin{proof}
$L$ is a direct summand of $M$ if and only if there exists a morphism $r:M\to L$ such that $r\circ i=1_{L}$. This is equivalent to prove that the function between the homs is surjective, i.e., $Hom_{A}(M,L)\epi Hom_{A}(L,L)$.\\
Since the forgetful functor $Hom_{A}(A,-)$ preserves epimorphisms, it is enough to prove that $\varphi:hom_{A}(M,L)\to hom_{A}(L,L)$ is an epimorphism.  The result comes from the following commutative diagram
$$
\xymatrix{
B\tens{A}hom_{A}(M,L)\ar_{\zeta_{1}\cong }[d]\ar^{B\x\varphi}[r] & B\tens{A}hom_{A}(L,L)\ar^{\cong \zeta_{2}}[d]\\
hom_{B}(B\tens{A}M,B\tens{A}L)\ar@{->>}[r]^{\psi}& hom_{B}(B\tens{A}L,B\tens{A}L).
}
$$
$\psi$ is an epimorphisms since $B\X{A}I$ is a direct summand of $B\X{A}M$ and the forgetful functor $Hom_{B}(B,-)$ reflects epimorphisms. Lemma \ref{lema 4.6} shows that $\zeta_{1}$ and $\zeta_{2}$ are isomorphisms, therefore $B\x \psi$ is an epimorphism and we get the result.
\end{proof}

\subsubsection*{\Large{Line Objects.}}

Next, following \cite{Saav, brandenburg}  we review the definition and properties of the objects that make possible the definition of the projective space. This kind of objects are the categorification of  rank one invertible sheaves over a scheme.

\begin{defi}[Invertible object]
If $\C$ is a symmetric monoidal category, $L\in \C$ is called invertible if there exists an object $L^{\vee}$ and an isomorphism $\delta:\1\to L^{\vee}\x L$. 
\end{defi}
Note that if $L$ is invertible then  $L\x -:\C\to\C$ is an equivalence with inverse $L^{\vee}\x -$.
\begin{obs}\label{invertible=proyectivo}\hfill
\begin{enumerate}
\item $\1$ is invertible and invertible objects are closed under tensor products. Isomorphisms classes of invertible objects form a group denoted $Pic(\C)$. For more details on $Pic(\C)$, see \cite{Maypic}. 
\item If $L$ is invertible, then for every isomorphism $\delta:\1\to L\x L^{\vee}$ there exists an isomorphism $\epsilon:L^{\vee}\x L\to \1$ satisfying the triangle axioms 
$$
\vcenter{
\xymatrix{
\1\x L\ar^{\delta\x 1}[r]\ar@{=}[d] & L\x L^{\vee}\x L\ar^{\epsilon\x 1}[ld]\\
	 \1\x L &
}
}
\cenesp
\vcenter{
\xymatrix{
\1\x L^{\vee}\ar^{\delta\x 1}[r]\ar@{=}[d] & L\x L^{\vee}\x L\ar^{\epsilon\x 1}[ld]\\
	\1\x L^{\vee} &
}}
$$
therefore $(L,L^{\vee},\epsilon,\delta)$ is a duality in $\C$.
\item If $L$ is invertible and $\1$ is projective, then $L$ is a projective object in $\C$. In fact, since $L$ is invertible, $hom_{\C}(L,-)$ is left adjoint to $hom_{\C}(L^{\vee},-)$, therefore it preserves colimits. As $Hom_{\C}(L,-)\cong Hom_{\C}(\1,hom_{\C}(L,-))$ and $Hom_{\C}(\1,-)$ preserves epimorphisms we have that $Hom_{\C}(L,-)$ preserves epimorphisms.
\end{enumerate}
\end{obs}

Now, for invertible objects there is a well defined signature
\begin{defi}[Signature]
Since $L\x -$ is an equivalence we have bijections $End_{\C}(\1)\cong End_{\C}(L)\cong End_{\C}(L\x L)$ then the signature is the endomorphism of $\1$ corresponding to the symmetry $\sigma_{L,L}:L\x L\to L\x L$ via that bijection. 
\end{defi}

\begin{defi}[Line object.]
$L\in \C$ is called a line object if it is invertible and its signature is the identity morphism. 
\end{defi}

\begin{obs}
An object $M$ in $\C$  is said to be {\it symtrivial} if $\sigma_{M,M}:M\x M\to M\x M$ is the identity arrow. Since the signature of an invertible object $L$ in $\C$ is the endomorphism associated to the symmetry of $L\x L$, then a line object is simply an invertible symtrivial object.
\end{obs}

\begin{proposition}\label{prop line objects}\hfill
\begin{enumerate}
\item Symtrivial objects are preserved by strong monoidal functors.
\item If $\C$ is cocomplete then $M\oplus N$ is symtrivial if and only if $M\x N=0$ and $M, N$ are symtrivial.
\item  Let $A$ be a faithfully flat commutative algebra in $\C$, $L\in \C$. If $A\x L$ is a line object in $\mod_{\C}(A)$ then  $L$ is a line object in $\C$ .
\item If $L$ is a line object in $\C$, then every epimorphism $\1\to L$ is an isomorphism.
\end{enumerate}
\end{proposition}
\begin{proof}
For details see \cite{brandenburg}.
\end{proof}
\begin{ej}\label{ejemplo invertibles}
\begin{enumerate}\hfill
\itemsep-0.5em 
\item Let $R$ be a commutative ring. Then $M$ is a line object in $\mod(R)$ if and only if $M$ is projective module of rank one. 
\item Let $X$ be a scheme, then the line objects in  $Qcoh(X)$ are precisely the invertible sheaves.
\end{enumerate}
\end{ej}

\begin{lemma}\label{product is line object}
Let $(SpecA_{{i}}\to SpecA)_{i\in I}$ be a finite Zariski open covering. Let the $A$-algebra $B=\prod_{i}A_{i}$. If for every $i\in I$,  $L_{i}$ is a line object in $\mod_{\C}(A_i)$ then $J=\prod_{i}L_{i}$ is a line object in $\mod_{\C}(B)$.
\end{lemma}
\begin{proof}
We claim that $J$ has an inverse in $\mod_{\C} (B)$ given by $J^{\vee}=\prod_{i}L_{i}^{\vee}$ with $L_{i}^{\vee}$ is the inverse of $L_{i}$ in $\mod_{\C}(A_{i})$ for all $i\in I$. If $m_{i},m_{i}^{\vee}$ denote the actions of $A_{i}$ on $L_{i}$ and $L_{i}^{\vee}$ respectively, we will prove the following two things:
\begin{itemize}
\item [i.] For every $i\in I$, $L_{i}\X{A_{i}}L_{i}^{\vee}\cong L_{i}\X{B}L_{i}^{\vee}$. let us consider the diagram with exact rows:
\begin{equation}\label{6.1eq2}
\begin{gathered}
\xymatrix{
L_{i}\X{A}B\X{A}L_{i}^{\vee}\ar^{\quad \bar r}[r]\ar^{1\x p_{i}\x 1}[d]&L_{i}\X{A}L_{i}^{\vee}\ar^{\pi}[r] \ar@{=}^{}[d]& L_{i}\X{B}L_{i}^{\vee}\ar@{-->}_{\varphi}[d]\\
L_{i}\X{A}A_{i}\X{A}L_{i}^{\vee}\ar^{\quad r}[r]\ar@/^{1pc}/^{1\x \lambda_{i}\x 1}[u] & L_{i}\X{A}L_{i}^{\vee}\ar^{\pi'}[r] &L_{i}\X{A_{i}}L_{i}^{\vee}\ar@{-->}@/_{1pc}/[u]_{\psi}
}
\end{gathered}
\end{equation}
where $\bar r=\bar{m_{i}}\x 1-1\x \bar{m_{i}}^{\vee}$, $r=m_{i}\x 1-1\x m_{i}^{\vee}$ and $\pi,\pi'$ the cokernel maps. As $\pi'\circ \bar r=\pi'\circ r\circ(1\x p_{i}\x 1)=0$, there exists an arrow $\varphi:L_{i}\X{B}L_{i}^{\vee}\to L_{i}\X{A_{i}}L_{i}^{\vee}$ sucht that $\varphi\pi=\pi'$.\\
 On the other hand, due to $$(1\x p_{i}\x 1)(1\x \lambda_{i}\x 1)=1,$$
 then 
 $$
 \pi\circ r=\pi\circ r(1\x p_{i}\x 1)\circ (1\x \lambda_{i}\x 1)=\pi\circ\bar r(1\x \lambda_{i}\x 1)=0
 $$
 so, there exists an arrow $\psi:L_{i}\X{A_{i}}L_{i}^{\vee}\to L_{i}\X{B}L_{i}^{\vee}$ satisfying $\psi\circ\pi'=\pi$. let us check they are inverse to each other.
 $$
 \psi\varphi\pi=\psi\pi'=\pi,\quad\quad \varphi\psi\pi'=\varphi\pi=\pi'
 $$
since $\pi,\pi'$ are epimorphisms we get $\psi\varphi=1$ and $\varphi\psi=1$.

\item [ii.]For every $i\neq j$, $L_{i}\X{B}L_{j}^{\vee}=0$: in order to prove this, we will prove that 
$$
r_{ij}=m_{i}\x 1-1\x m_{j}^{\vee}:L_{i}\X{A}B\X{A}L_{j}^{\vee}\to L_{i}\X{A}L_{j}^{\vee}
$$
is an epimorphism, thence its cokernel $L_{i}\X{B}L_{j}^{\vee}$ would be the zero object. For this, consider for every $i\in I$, the morphism $\lambda^{(i)}:A\to B$ given by  $(0,\cdots,\eta_{i},0\cdots,)$ with $\eta_{i}:A\to A_{i}$ the unit of $A_{i}$ as an $A$-algebra, in the $i$-th position
%and $\lambda_{j}=0:A\to A_{j}$ if $j\neq i$ for every $i\in I$ then we have that 
$$
 r_{ij}(1\x\lambda_{i}\x 1):L_{i}\X{A}A\X{A}L_{j}^{\vee}\to L_{i}\X{A}B\X{A}L_{j}^{\vee}
$$
is the identity arrow for $i\neq j$, this means that $r_{ij}$ is an epimorphism for $i\neq j$.
%$$
%\xymatrix{L_{i}\X{A}B\X{A}L_{j}^{\vee}\ar@<-1ex>_{r\ \ \ \  \ \  \ \ \quad}[r]& L_{i}\X{A}A\X{A}L_{j}^{\vee}\ar@<-1ex>_{s\quad\quad\ \ \ }[l]\cong L_{i}\X{A}L_{j}^{\vee}}
%$$
\end{itemize}
combining i. and ii. we have that 
\begin{align*}
J\X{B}J^{\vee}\cong &\prod_{i,j}coKer\left(L_{i}\X{A}B\X{A}L_{j}^{\vee}\to L_{i}\X{A}L_{j}^{\vee}\right)\\
\cong&\prod_{i}coKer\left(L_{i}\X{A}B\X{A}L_{i}^{\vee}\to L_{i}\X{A}L_{i}^{\vee}\right)\\
\cong& \prod_{i} L_{i}\X{B}L_{i}^{\vee}\cong\prod_{i} L_{i}\X{A_{i}}L_{i}^{\vee}\cong\prod_{i}A_{i}=B
\end{align*}
Now we show that $J$ is a symtrivial object in $\mod_{\C}(B)$ provided that each $L_{i}$ is symtrivial in $\mod_{\C}(A_{i})$ for all $i$. let us denote $\sigma, \sigma^{i},\sigma^{B}$ the symmetries in $\mod_{\C}(A)$,$\mod_{\C}(A_{i})$, $\mod_{\C}(B)$ respectively and consider the following diagram, where unadorned tensor means $\X{A}$
\begin{equation*}
\xymatrix{
L_{i}\x A_{i}\x L_{i}\ar^{}[r] \ar_{1\x \lambda_{i}\x 1}[d] & L_{i}\x L_{i} \ar^{}[r]\ar@{=}[d] &L_{i}\X{A_{i}}L_{i}\ar^{\psi}[d]\ar@/^{2.5pc}/^{\sigma^{i}_{L_{i},L_{i}}}[ddd]\\
L_{i}\x B\x L_{i}\ar^{}[r] \ar_{(\sigma_{B,L_{i}}\x 1)(1\x \sigma_{L_{i},L_{i}})(\sigma_{L_{i},B}\x 1)}[d] & L_{i}\x L_{i} \ar^{}[r]\ar^{\sigma_{L_{i},L_{i}}}[d] &L_{i}\X{B}L_{i}\ar^{\sigma^{B}_{L_{i},L_{i}}}[d]\\
L_{i}\x B\x L_{i}\ar_{1\x p_{i}\x 1}[d] \ar[r]  &  L_{i}\x L_{i} \ar@{=}^{}[d]\ar^{}[r] & L_{i}\X{B}L_{i}\ar^{\varphi}[d]\\
 L_{i}\x A_{i}\x L_{i}\ar^{}[r]&L_{i}\x L_{i}\ar^{}[r] &L_{i}\X{A_{i}}L_{i}  
}
\end{equation*}
with $\varphi, \psi$ defined as in \ref{6.1eq2} with $L_{i}^{\vee}=L_{i}$.  By naturality of $\sigma$ and the identity
$$
(1\x p_{i}\x 1)(1\x \lambda_{i}\x 1)=1$$
 we have that 
$$
(1\x p_{i}\x 1)(\sigma_{B,L_{i}}\x 1)(1\x \sigma_{L_{i},L_{i}})(\sigma_{L_{i},B}\x 1)(1\x \lambda_{i}\x 1)=(\sigma_{A_{i},L_{i}}\x 1)(1\x \sigma_{L_{i},L_{i}})(\sigma_{L_{i},A_{i}}\x 1)
$$
consequently $\varphi~\sigma^{B}_{L_{i},L_{i}}\varphi^{-1}=\sigma^{i}_{L_{i},L_{i}}=1$ which implies $\sigma^{B}_{L_{i},L_{i}}=1$. \\

So far, we have proved that each $L_{i}$ is a symtrivial object in $\mod_{\C}(B)$. To finally get the result, we use the fact that $L_{i}\X{B}L_{j}=0$ for every $i\neq j$ and proposition \ref{prop line objects} (2), so $\bigoplus_{i}L_{i}$ is symtrivial in $\mod_{\C}(B)$
\end{proof}

\section{{The scheme $\Proj{\C}$}}\label{esquema proyectivo}
%In the sequel $\C$ is a symmetric monoidal abelian category  such that $\1_{\C}$  is compact.\\
As a motivation for the definition of the projective space, we first recall a characterisation of the functor of points of the scheme $\Proj{\ent}$. Let us denote $Mor(X,Y)$ the set of morphisms in the category of schemes $Sch$, then we have that:

\begin{theorem}[See \cite{geomofschemes}]
For any ring $A$, 
\begin{eqnarray*}
&\text{Mor}(Spec A,\Proj{\ent})=\{L\subset A^{n+1}~:L~\text{is a locally rank 1 direct summand of }A^{n+1}\}\\
&\cong\{\text{invertible}~A-\text{modules}~P~\text{with an epimorphism}~A^{n+1}\to P\}/\{\text{isomorphisms}\}
\end{eqnarray*}
where by invertible module we mean a finitely generated, locally free $A$-module of rank 1 and an isomorphism from $\varphi:A^{n+1}\to P$ to $\varphi':A^{n+1}\to P$ is an automorphism $\alpha:P\to P$ such that $\alpha\varphi=\varphi'$.\\
Moreover, for any scheme $X$, one has the natural bijection
$$
Mor(X,\Proj{\ent})=\{\text{Invertible sheaves $P$ in }\mathcal Qcoh(X)~\text{with an epimorphism}~\mathcal O_{X}^{n+1}\to P\}/\{iso\}
$$
\end{theorem}

Having this characterisation in mind and by example \ref{ejemplo invertibles} and remark \ref{invertible=proyectivo} item iii), we define the projective space relative to the category $\C$ as the functor $\Proj{\C}:Aff_{\C}^{op}\to\mathcal Ens$, as follows:

\begin{defi}\label{def projectivo}[Relative Projective Scheme] Let $n\geq 1$ a fixed integer. For every affine scheme $Spec(A)$ we define $\Proj{\C}(A)$ to be the set of submodules $L$ of $A^{n+1}$ satisfying 
\begin{itemize}
\item $L$ is a line object in $\mod_{\C}(A)$
\item For the monomorphism ${\bf x}:L\to A^{n+1}$, there exists a retraction $A^{n+1}\to L$, this is, $L$ is a direct summand of $A^{n+1}$. 
\end{itemize}
For every morphism $Spec(B)\to Spec(A)$ in $Aff_{\C}$, the function $\Proj{\C}(A)\to \Proj{\C}(B)$ assigns to $L\in \Proj{\C}(A)$ the corresponding direct summand $B\X{A}L\mono B^{n+1}$.
\end{defi}
 Note that $B\X{A}L$ is a line object in $\mod_{\C}(B)$ since line objects are preserved by strong monoidal functors.
\begin{obs}
Note that for every $A\in Comm(\C)$ a pair $(L,{\bf x})$ in $\Proj{\C}(A)$ is a subobject, that is, a class of monomorphisms of $A^{n+1}$, where $(L_{1}{\bf x}_{1}),~(L_{2},{\bf x}_{2})$ represent the same element subobject, if there exists an isomorphism $\lambda:L_{1}\to L_{2}$ such that the diagram commutes
$$
\xymatrix@C=0.8ex{
L_{1}\ar_{\lambda}[rd]\ar@{^{(}->}^{{\bf x}_{1}}[rr] & & A^{n+1}\\
		& L_{2}\ar@{^{(}->}_{{\bf x}_{2}}[ru]
}
$$
Since $L$ is an invertible object we have that $Aut(L)\cong Aut(A)$, therefore the equivalence relation is given by scalar multiplication by invertible elements in $A$.  So if we think of the pair $(L,{\bf x})$ as a vector in $A^{n+1}$, its class in $\Proj{\C}(A)$ represents the ``line'' in $A^{n+1}$. This is kind of the intuition one has of the classical projective space.
\end{obs}

%let us consider the functor $\Proj{\C}:Comm(\C)\to Ens$
%\begin{align*}
%	\Proj{\C}(A)=&\{L\mono A^{n+1} \text{~ with L a line object in ~}\mod_{\C}(A)\}/\sim \\ \notag
%	 %\cong &\{q:A^{n+1}\epi L \text{~with L locally free rank 1 object in~}\mod_{\C}(A)\}/\equiv \notag
%\end{align*}
%and for every morphism $A\to B$ we associate the corresponding sub-$B$ module $B\X{A}L\mono B^{n+1}$.
%$$
%\xymatrix@-1.2pc{
%	%\Proj{\C}(A)\ar^{f^{*}}[r] & \Proj{\C}(B)  \\
%	 q:A^{n+1}\epi L \ar@{|->}^{}[rr] & &f^{*}(q):B^{n+1}\epi L\X{A}B
%}
%$$
% with ''inverse'' $L^{\vee}\X{A}B$ since the functor $-\X{A}B:\mod_{\C}(A)\to\mod_{\C}(B)$ is a tensor functor and symmtrivial objects are preserved by these functors

\begin{theorem}\label{teorema projectivo}
Let $\C$ be an abelian strong relative context. %  and reflexes epimorphisms un functor conservativo refleja los limites o colimites que preserva ver prop en apendice
Then the presheaf  $\Proj{\C}$ is a $\C$-scheme. 
\end{theorem}

\begin{proof}
Let us check the sheaf condition in the Zariski topology: Let $\{SpecA_{i}\to SpecA\}_{i}$ be a Zariski covering, we have to prove the exactness of the sequence 
\begin{equation}\label{condicion de haz}
\xymatrix{
\Proj{\C}(A)\ar^{}[r] & \prod_{i}\Proj{\C}(A_{i})\ar@<0.5ex>[r]^{} \ar@<-0.5ex>[r]_{}& \prod_{i,j}\Proj{\C}(A_{ij}).
}
\end{equation}
Let $L\in \Proj{\C}(A)$, by the equivalence given in  (\ref{M stack}) %and Corollary \ref{subcanonical} 
the following sequence is exact
$$
\xymatrix{
L\ar^{}[r] & \prod_{i}L_{i}\ar@<0.5ex>[r]^{} \ar@<-0.5ex>[r]_{}& \prod_{i,j}L_{ij}
%A^{n+1} \ar^{}[r]          &  \prod_{i}A_{f_{i}}^{n+1}\ar@<0.5ex>[r]^{} \ar@<-0.5ex>[r]_{}   & \prod_{i,j}A_{f_{i}f_{j}}^{n+1}
}
$$ 
 %(LA EXACTITUDE TAMBIEN PROVIENE Y ES MEJOR ARGUMENTO, DE COROLARIO 2.11 DE TV QUE DICE QUE LA TOPOLOGIA FPQC ES SUBCANONICA Y SATISFACE DESCENSO PARA LOS MODULOS)) \\
 
%it means that  $P=\lim\left(\xymatrix{ \prod_{i}P_{f_{i}}\ar@<0.5ex>[r]^{} \ar@<-0.5ex>[r]_{}& \prod_{i,j}P_{f_{i}f_{j}}}\right)$
then $L$ is determined by $L_{i}\in \Proj{\C}(A_{i})$ therefore $\Proj{\C}(A)$ is a cone of the diagram.   %in (\ref{condicion de haz})
 \\
Now we have to check that $\Proj{\C}(A)$ is universal. To see this, consider the compatible family $(L_{i})_{i}\in\prod_{i}\Proj{\C}(A_{i})$. The compatibility says that we have a family of isomorphisms
$$
\xymatrix{
L_{i}\X{A}A_{j}\ar^{\cong}_{\theta_{ij}}[r]\ar@{^{(}->}[rd] & L_{j}\X{A}A_{i}\ar@{^{(}->}[d]\\
   & A_{ij}^{n+1},
}
$$
let us prove that $(L_{i},\theta_{i,j})_{i,j}$ is a descent data, that is, $\theta_{ij}$ satisfies the cocycle condition $\theta_{j,k}\circ\theta_{i,j}=\theta_{i,k}$ in $\mod(A_{i,j,k})$. In fact, the following diagram of sub-objects of $A_{i,j,k}^{n+1}$
$$
\xymatrix{
L_{i}\X{A}A_{j}\X{A}A_{k}\ar@/^{2.5pc}/[rr]^{\theta_{i,k}\x A_{j}}\ar^{\theta_{i,j}\x A_{k}}[r] & L_{j}\X{A}A_{i}\X{A}A_{k}\ar^{\theta_{j,k}\x A_{i}}[r]& L_{k}\X{A}A_{i}\X{A}A_{j}
}
$$
says that the two arrows coincide since between two subobjects there is at most one arrow. Again by the equivalence given in (\ref{M stack}), we have that the descent data $(L_{i},\theta_{i,j})_{i,j}$ defines an $A$-module $L$ as the limit of the diagram
$$
\xymatrix@R=0.5pc{
   &   L_{i}\ar^{}[r] & L_{i}\x A_{j}\ar^{\theta_{i,j}}[dd] \\
L \ar^{}[ru]\ar^{}[rd]   &     &\\
     & L_{j}\ar^{}[r]    &  L_{j}\x A_{i} .
}
$$
%  Let  $J$ to be the $B$-module  $\prod_{i}L_{i}$ with $B=\prod_{i}A_{f_{i}}$. By abuse of notation we denote $p(J)=q(J)$ meaning that the inclusions $in_{1}, in_{2}$ coincide, i.e.,  the  following diagram commutes
%$$
%\xymatrix{
%J\ar@{^{(}->}[r]^{} & \prod_{i}A_{f_{i}}^{n+1}\ar@<0.5ex>[r]^{in_{1}} \ar@<-0.5ex>[r]_{in_{2}} & \prod_{i,j}A_{f_{i}f_{j}}^{n+1}
%}
%$$ 
%we denote this submodule $K=in_{i}(J)\in \mod_{}(B\X{A}B)$. It follows that if $in_{i}$ denotes the $i$-th inclusion of $B$ in $B\X{A}B\X{A}B$ ($i=1,2,3$), the $B\X{A}B\X{A}B$-submodule $M$ of $(B\X{A}B\X{A}B)^{n+1}$ generated by $in_{i}(J)$  is independent of $i$. 
%$$
%\xymatrix{
%\quad J\quad \ar@{^{(}->}^{}[d] \ar@<0.5ex>[r]^{} \ar@<-0.5ex>[r]_{} & \quad K\quad  \ar@{^{(}->}^{}[d]\ar@<1.0ex>[r]^{}\ar^{}[r] \ar@<-1.0ex>[r]_{}&\quad M\quad\ar@{^{(}->}^{}[d]\\
%B^{n+1}  \ar@<0.5ex>[r]^{} \ar@<-0.5ex>[r]_{}              &  (B\X{A}B)^{n+1}   \ar@<1.0ex>[r]^{}\ar^{}[r] \ar@<-1.0ex>[r]_{}&  (B\X{A}B\X{A}B)^{n+1}
%}
%$$

%By lemma \ref{lema3.14 Demazure} the $A$-module $L$ 
%$$L=lim(\xymatrix{J\ar@<0.5ex>[r]^{} \ar@<-0.5ex>[r]_{}& K\ar@<1.0ex>[r]^{}\ar^{}[r] \ar@<-1.0ex>[r]_{}&M})$$
To prove that $L\in\Proj{\C}(A)$, consider the product algebra $B=\prod_{i}A_{i}$, note that $B$ is a faithfully flat $A$-algebra as $\mod_{\C}(B)\cong \prod_{i}\mod_{\C}(A_{i})$ and the functor $-\X{A}B$ is naturally isomorphic to $\bigtimes_{i}(-\X{A}A_{i})$.  Now take the $B$-module $L\x B$, then we have that
$$
L\x B\cong\prod_{i}L\x A_{i}\cong\prod_{i}L_{i}.
$$
By Lemma \ref{product is line object}, $L\x B$ is a line object in $\mod_{}(B)$ therefore by proposition \ref{prop line objects} %{PROP 4.10.3 Y 4.10.5 DE BRANDEN} 
we have that $L$ is a line object in $\mod_{}(A)$. Finally by Lemma \ref{lema 3.15 Demazure}, $L$ is a direct summand of $A^{n+1}$.\\

$\Proj{\C}$ is covered by the affine open sub-functors $U_{i}$ for $i=1,\cdots n+1$ 
\begin{equation}\label{abiertos afines}
U_{i}(A)=\{L\in\mod_{\C}(A):\xymatrix@-1pc {L\ar@{^{(}->}^{{\bf{x}}\ \ }[r] & A^{n+1}\ar^{\pi_{i}}[r] & A}\ \ \text{$\pi_{i}\circ{\bf{x}} ~$is an isomorphism}\}.
\end{equation}
 %We have to prove that this subfunctors are affine schemes, in fact they are isomorphic to the functor $Hom(\1[x_{1},\cdots x_{n},-]$ which is the affine $n$-space.
\textbf{Representability of the subfunctors $U_{i}$:} let us fix the index $i$. Given any element $(L,\mathbf x)\in U_{i}(A)$, we identify $L$ with  $A$ as submodules of $A^{n+1}$ via the isomorphism $\pi_{i}\mathbf x:L\cong A$, then we obtain $\tilde{\mathbf x}=\mathbf x(\pi_{i}\mathbf x)^{-1}:A\to A^{n+1}$, this means that $(L,\mathbf x)=(A,\mathbf{\tilde{x}})$ as subobjects of $A^{n+1}$. Since $\pi_{i}\tilde {\mathbf x}=1$, $\tilde{ \mathbf x}$ is completely determined by specifying the morphisms $\pi_{j}\tilde{\mathbf x}:A\to A$ for $j=1,\cdots n+1$ and $j\neq i$, i.e., the functor $U_{i}$ is isomorphic to the functor 
\begin{align*}
	A\mapsto \prod_{\substack{j=1\\ j\neq i}}^{n+1}\Hom{A}{A}{A}&\cong \prod_{j=1}^{n}\Hom{\1_{\C}}{A}{\C}\\
	& \cong\Hom{\1^{n}}{A}{\C}\cong\Hom{\1[x_{1},\cdots x_{n}]}	{A}{Comm(\C)}=\mathbb A^{n}_{\C}(A),
\end{align*}
therefore $U_{i}$ is representable by an affine scheme.
%where in the first isomorphism we used that the functor $M\mapsto M\x A$ from $\C$ to $\mod_{\C}(A)$ is left adjoint to the forgetful functor. See proposition \ref{I-adjuncion de A-mod y olvido}.\\

\textbf{The sub functors $U_{i}$ are Zariski open immersions:} let us see that for affine scheme $h_{A}$ and any morphism $h_{A}\to \Proj{\C}$,  the pullback $h_{A}\times_{\Proj{\C}}U_{i}$ is a Zariski immersion of $h_{A}$.\\
By Yoneda's Lemma the morphism $h_{A}\to \Proj{\C}$ corresponds to $(L,\mathbf x)$ in $\Proj{\C}(A)$. Consider the pullback 
$$
\xymatrix@-0.7pc{
	V_{i}=h_{A}\times_{\Proj{\C}}U_{i} \ar^{}[r] \ar^{}[d]& U_{i}\ar@{>->}^{}[d]\\
	h_{A}\ar^{}[r]   & \Proj{\C} .
}
$$ 
%Since $U_{i}$ is representable we have that 
%$$
%V_{i}(B)=\{(f,L)~|~ f:A\to B,~ L\in U_{i}(B)~\text{and }f^{*}(L^{\vee})=L\}
%$$
%that is to say 
Now, an element in $V_{i}(B)$ is the same as a morphism $f:A\to B\in Comm(\C)$ such that $\xymatrix@-1.1pc{B\X{A}L\ar^{}[r] & B^{n+1}\ar^{}[r]& B}$ is an isomorphism. If $I_{i}$ denotes the ideal in $A$ defined by the image of $\pi_{i}\circ\mathbf x$,  then tensoring the factorization of this arrow with $B$, we get a diagram 
\begin{figure}[H]
\begin{center}
\begin{tabular}{cc}
$
\xymatrix{L\ar[rr]^{\pi_{i}\circ\mathbf x}\ar@{->>}[rd]_{p} &  &A\\
  &  I_{i}\ar@{^{(}->}[ru]_{j}
}
$
&
$
\xymatrix@-0.8pc{
B\X{A}L\ar^{B\x(\pi_{i}\circ \mathbf x)}[rr]\ar_{B\x p}[rd] & & B\X{A}A\cong B.\\
                       & B\X{A}I_{i}\ar_{m_{B}(B\x f\circ j)}[ru]  &
}
$

\end{tabular}
\end{center}
\end{figure}
We have that all the arrows in the triangle on the right are isomorphisms. On the other hand, consider the ideal $BI_{j}$, which by definition is the image of $m_{B}(B\x f\circ j)$ then we have 
%$$
%\xymatrix@-0.8pc{
%B\X{A} I_{j}\ar^{\ \ \ m_{B}(B\x f\circ j)}[rr]\ar@{->>}[rd]_{} &  & B\\
%  & BI_{j}\ar@{^{(}->}[ru]^{}    &
%}
%$$
$B\X{A}I_{i}\cong BI_{i}$. This means that $V_{i}$ is contained in the complementary open subscheme associated to the ideal $I_{i}$. Let us see that the complementary open $U_{I_{i}}$ defined by the ideal $I_{i}$ is contained in $V_{i}$ . Let $f\in U_{I_{i}}(B)$, i.e., $f:A\to B$ satisfies that the induced ideal $BI_{i}$ is isomorphic with $B$. Then we have that, by the triangle in the right, $B\X{A}L\to B$ is an epimorphism and $B\X{A}L$ is a line object in $\mod(B)$, so tensoring this epimorphism with the inverse of $B\X{A}L$, we have again an epi $B\epi B\X{A}L^{\vee}$ which by Proposition \ref{prop line objects}(4) is an isomorphism in $\mod(B)$, then tensoring again with the inverse we get $\xymatrix{B\X{A}L\ar^{\quad\cong}[r]& B}$. This means that $f\in V_{i}(B)$.  Finally by Lemma \ref{complementary open}, $V_{i}\subset Spec A$ is a Zariski open, so is $U_{i}\subset \Proj{\C}$.\\

\textbf{The family $(U_{i})_{i}$ is an affine Zariski open covering}: We have to prove that 
$$
\coprod_{i}U_{i}\to \Proj{\C}
$$
 is an epimorphism of sheaves. By lemma \ref{epi of sheaves} is enough to prove that $\coprod_{i}U_{i}(\mathbb K)\to \Proj{\C}(\mathbb K)$ is surjective for every field $\mathbb K\in Comm(\C)$.

Let $L\in \Proj{\C}(\mathbb K)$, i.e., ${\bf x}:L\mono \mathbb K^{n+1}$, then there exists an index $j$ such that the arrow $\pi_{j}\circ {\bf x}:L\to \mathbb K$ is non zero but then the image ideal $I_{j}$ in  $\mathbb K$  must be exactly $\mathbb K$ thus we have an epi $L\epi \mathbb K$. Since $L$ is a line object we have $L\X{\mathbb K}L^{\vee}\cong \mathbb K$  therefore $\mathbb K\epi L^{\vee}$, thus by Proposition \ref{prop line objects}(4),  $\mathbb K\cong L^{\vee}$ so $\mathbb K\cong L$.
\end{proof}

%\begin{ej}
%If $\C=\mod(\ent)$ then $\Proj{\C}=\Proj{\mbox{\tiny{$\ent$}}}$
%\end{ej}

\begin{defi}
 $\M$ is a symmetric monoidal category, $A\in Comm(\M)$ and $\C=\mod_{\M}(A)$ then we define $\Proj{A}:=\Proj{\C}$.
\end{defi}

%\subsection*{\Large{Alternative Definition for the Projective Space}}
Now we give another definition of the relative projective space in terms of quotients instead of submodules. This definition is somehow dual to the one given in definition \ref{def projectivo} and we show that these two definitions are in fact equivalent. 
\begin{defi}\label{projectivo con epis}
Let $n\geq 1$ a fixed integer. For every affine scheme $Spec(A)$ we define $\overline{\Proj{\C}}(A)$ to be the set of quotients $L$ of $A^{n+1}$ with $L$ a line object in $\mod_{\C}(A)$. 
For every morphism $Spec(B)\to Spec(A)$, the function $\overline{\Proj{\C}}(A)\to \overline{\Proj{\C}}(B)$ assigns to $L\in\overline{\Proj{\C}}(A)$ the corresponding epimorphism $B^{n+1}\epi B\X{A}L$.
\end{defi}
As before, $B\X{A}L$ is a line object in $\mod_{\C}(B)$ since line objects are preserved by strong monoidal functors.
\begin{theorem}\label{teorema projectivo2}
If $\C$ is an abelian strong relative context then $\overline{\Proj{\C}}$ is a $\C$-scheme.
\end{theorem}
The proof of this theorem is quite similar to its analogous result \ref{teorema projectivo}, however by the very definition we will not need Lemmas \ref{lema 4.6} and \ref{lema 3.15 Demazure}.
\begin{proof}
The sheaf condition is proven similarly as we did for $\Proj{\C}$. Let $(A\to A_{i})_{i\in I}$ be a Zariski covering for $Spec A$, we have to check the exactness of the diagram 
$$
\xymatrix{
\overline{\Proj{\C}}(A)\ar^{}[r] & \prod_{i}\overline{\Proj{\C}}(A_{i})\ar@<0.5ex>[r]^{} \ar@<-0.5ex>[r]_{}& \prod_{i,j}\overline{\Proj{\C}}(A_{ij}).
}
$$
We proceed as before to show that $\overline{\Proj{\C}}(A)$ is a cone for the diagram. To show that is universal consider the compatible family $(L_{i})_{i}\in\prod_{i}\overline{\Proj{\C}}(A_{i})$. The compatibility says that we have a family of isomorphisms $\theta_{ij}$ making the diagram commute
$$
\xymatrix{
 & A_{ij}^{n+1}\ar@{->>}[ld]\ar@{->>}[d]\\
L_{i}\X{A}A_{j}\ar^{\cong}_{\theta_{ij}}[r] & L_{j}\X{A}A_{i},
}
$$
$(L_{i},\theta_{i,j})_{i,j}$ is a descent data, that is, $\theta_{ij}$ satisfies the cocycle condition $\theta_{j,k}\circ\theta_{i,j}=\theta_{i,k}$ in $\mod(A_{i,j,k})$. 
%In fact, the following diagram of sub-objects of $A_{i,j,k}^{n+1}$
%$$
%\xymatrix{
%L_{i}\X{A}A_{j}\X{A}A_{k}\ar@/^{2.5pc}/[rr]^{\theta_{i,k}\x A_{j}}\ar^{\theta_{i,j}\x A_{k}}[r] & L_{j}\X{A}A_{i}\X{A}A_{k}\ar^{\theta_{j,k}\x A_{i}}[r]& L_{k}\X{A}A_{i}\X{A}A_{j}
%}
%$$
%says that the two arrows coincide since between two subobjects there is at most one arrow. 
Thus by the equivalence given in (\ref{M stack}), we have that the descent data $(L_{i},\theta_{i,j})_{i,j}$ defines an $A$-module $L$ as the limit of the diagram
$$
\xymatrix@R=0.5pc{
   &   L_{i}\ar^{}[r] & L_{i}\x A_{j}\ar^{\theta_{i,j}}[dd] \\
L \ar^{}[ru]\ar^{}[rd]   &     &\\
     & L_{j}\ar^{}[r]    &  L_{j}\x A_{i} .
}
$$
This $L$ is a line object by propositions \ref{prop line objects}, \ref{product is line object}. Finally to see that $A^{n+1}\to L$ is an epimorphism we use the fact that for every $i\in I$, we have the family of epimorphisms $A_{i}^{n+1}\epi L_{i}\cong A_{i}\X{A}L$, since the family of functors $A_{i}\X{A}-$ is jointly conservative we get the result.\\
We now prove that $\overline{\Proj{\C}}$ has an affine Zariski open covering. For this, we define for $i=0,\dots n$
$$
\overline{U_{i}}(A)=\{(L,{\bf x}),\ \ \text{such that the composition }\xymatrix@C=3ex{A\ar^{\lambda_{i}}[r]&A^{n+1}\ar@{->>}^{{\bf x}}[r] & L}\text{ is an isomorphism} \}
$$
the isomorphism ${\bf x}\lambda_{i}$ occurs in $\mod_{\C}(A)$. We will check the representability of these functors by showing that $U_{i}\cong \overline{U_{i}}$ for $i=0,\dots n$.\\
In fact, for every affine scheme $Spec A$, we will define a bijection $U_{i}(A)\longleftrightarrow\overline{U_{i}}(A)$. First let us make a simplification: if $(L,{\bf x})$ belongs to $U_{i}(A)$ we can make the identification $L\cong A$ as subobjects of $A^{n+1}$, we will denote the pair $(A,{\bf x})$ in $U_{i}(A)$. The same goes for a pair $(L,{\bf y})$ in $\overline{U_{i}}(A)$. let us fix the index $i$:
$$
\xymatrix@R=2ex{
U_{i}(A)\ar^{\varphi}[r] & \overline{U_{i}}(A)\\
(A,{\bf x})\ar@{|->}[r] 	&  (A,{\bf y})
}
$$ 
with ${\bf y}$ defined by the following: for every $j=0,\dots n$, the diagram commutes
$$
\xymatrix{
A\ar^{\pi_{j}{\bf x}}[rd] \ar^{\lambda_{j}}[d]& \\
A^{n+1}\ar@{-->}[r]_{\exists! {\bf y}} & A
}
$$
since ${\bf y}\lambda_{i}=\pi_{i}{\bf x}$ is an isomorphism we have that ${\bf y}$ is an epimorphism, even more that $(A,{\bf y})$ is in $\overline{U_{i}}(A)$.\\
For the arrow in the other direction:
\begin{equation}\label{iso de abiertos}
\xymatrix@R=2.5ex{
\overline{U_{i}}(A)\ar^{\psi}[r] & U_{i}(A)\\
(A,{\bf y})\ar@{|->}[r] 	&  (A,{\bf x})
}
\end{equation}
with ${\bf x}$ defined analogously by the following diagram for every $j=0,\dots n$:
$$
\xymatrix{
& A \\
A\ar^{{\bf x}\lambda_{j}}[ru]\ar@{-->}_{{\bf x}}[r]& A^{n+1},\ar[u]_{\pi_{j}}
}
$$
as $\pi_{i}{\bf x}={\bf y}\lambda_{i}$ is an isomorphisms it says that $(A,{\bf x})$ is in $U_{i}(A)$.\\
We will check that $\psi\varphi=1$ the other one is similar. 
$$
\psi\varphi(A,{\bf x})=\psi(A,{\bf y})=(A,\tilde{\bf x})
$$
with $\pi_{j}\tilde{\bf x}={\bf y}\lambda_{j}=\pi_{j}{\bf x}$ for all $j=0,\dots n$, then $\tilde{\bf x}={\bf x}$.\\
The next step is to prove that every $\overline{U_{i}}$ is a Zariski open immersion of $\overline{\Proj{\C}}$. Again, we will show that for any affine scheme $h_{A}$ and morphism $h_{A}\to \overline{\Proj{\C}}$, the pullback 
$$
\xymatrix{\overline{V_{i}}\ar^{}[d]\ar[r]& \overline{U_{i}}\ar[d]\\
h_{A}\ar^{}[r] &  {\overline{\Proj{\C}}}
}
$$
is a Zariski open in $h_{A}$. We proceed as we did before, that is, we show that the subfunctor $V_{i}$ is equivalent to the complementary open subscheme of $h_{A}$ associated to an ideal $I$ of $A$. For $B\in Comm(\C)$, $\overline{V_{i}}(B)$ consists of morphisms $f:A\to B$ in $Comm(\C)$ such that the if $(L,{\bf x})$ is in $\overline{U_{i}}(A)$, the induced morphism
$$
B\X{A}A\to B\X{A}A^{n+1}\to B\X{A}L
$$
is an isomorphism. Take the dual morphism (as they are dualizable objects in $\mod_{\C}(A)$) of the composition $\xymatrix{A\ar^{\lambda_{i}}[r]&A^{n+1}\ar@{->>}[r]^{{\bf x}}& L}$ and take its image ideal $I$, as we see in the factorization diagram
$$
\xymatrix@R=3ex{
L\ar^{{\bf x}^{\vee}\quad}[r]\ar@{->>}[rd] &   A^{n+1}\ar^{\lambda_{i}^{\vee}}[r]&A\\
	& I\quad\ar@{>->}_{j}[ru] &
}
$$ 
applying the functor $B\X{A}-$ we obtain the diagram
$$
\xymatrix@R=3ex{
B\X{A}L\ar^{{1\x\bf x}^{\vee}\quad}[r]\ar@{->>}[rd] &   B\X{A}A^{n+1}\ar^{1\x \lambda_{i}^{\vee}}[r]&B\X{A}A\\
	& B\X{A}I\quad\ar_{1\x j}[ru] &
}
$$
the morphism in the top of the triangle is an isomorphism as it is the dual of an isomorphism, then the epimorphism $B\X{A}L\epi B\X{A}I$ is a monomorphism, therefore all arrows in the triangle are isomorphisms, this means that $\overline{V_{i}}(B)\subset U_{I}$ where $U_{I}$ denotes the complementary open subscheme associated to the ideal $I$. The other inclusion is obtained similarly.\\
Finally to show that the family $(\overline{U_{i}})_{i=0,\dots n}$ is a covering we will prove that for every field $\mathbb K\in Comm(\C)$ we have a surjection $\coprod_{i}\overline{U_{i}}(\mathbb K)\to \overline{\Proj{\C}}(\mathbb K)$. In fact, let $(L,{\bf x})$ in $\overline{\Proj{\C}}(\mathbb K)$, then there exists an index $j$ such that ${\bf x}\lambda_{j}:\mathbb K\to L$ is the non zero arrow, then taking its dual morphism 
$$
\xymatrix@R=2ex{L^{\vee}\ar^{({{\bf x}\lambda_{i}})^{\vee}}[r] & \mathbb K
}
$$
we have that this morphism must be an epimorphism since its image is an ideal in $\mathbb K$ and $\mathbb K$ is simple. After we tensor this epi with $L$ we get an epimorphism $\mathbb K\epi L$ which by Corollary \ref{epi a un line object es iso} is an isomorphism. 
\end{proof}
\begin{theorem}\label{los projectivos son isomorfos}
$\Proj{\C}$  and  $\overline{\Proj{\C}}$ are isomorphic as $\C$-schemes.
\end{theorem}
\begin{proof}
Since the category of $\C$- schemes is a full subcategory of $Sh(Aff_{\C})$, we will prove the isomorphism as sheaves. Let us define for every $A\in Comm(\C)$ a function 
$$
\xymatrix@R=2ex{
\overline{\Proj{\C}}(A)\ar^{\Psi}[r] & {\Proj{\C}}(A)\\
A^{n+1}\overset{{\bf x}}{\longrightarrow} L\quad \ar@{|->}[r] &\quad L^{\vee}\overset{{\bf x^{\vee}}}{\longrightarrow} A^{n+1} .
}
$$
Since ${\bf x}$ is an epimorphism, ${\bf x}^{\vee}$ is a monomorphism. As $L$ is invertible, hence projective, there exists a section $s$ for ${\bf x}$, then $r=s^{\vee}$ is a retraction for ${\bf x}^{\vee}$ thence $(L^{\vee},{\bf x}^{\vee})$ is an element in $\Proj{\C}(A)$. let us see that $\Psi$ is injective. Take ${\bf x}_{i}:A^{n+1}\epi L_{i}$, $i=1,2$ two elements in $\overline{\Proj{\C}}(A)$ such that their images coincide, then we have that $L^{\vee}_{1}$ and $L^{\vee}_{2}$ are isomorphic as subobjects of $A^{n+1}$ as it's seen in the diagram
$$
\xymatrix{L_{1}^{\vee} \ar^{{\bf x}_{1}^{\vee}}[r]\ar_{\cong}[d] &A^{n+1}\\
L_{2}^{\vee} \ar_{{\bf x}_{2}^{\vee}}[ru] &
}
$$
by dualizing we obtain that $L_{1}$ and $L_{2}$ are isomorphic as quotients of $A^{n+1}$, therefore they represent the same element in $\overline{\Proj{\C}}(A)$.\\
To see that $\Psi$ is an epimorphism, we check that for every $i$, the following diagram commutes:
$$
\xymatrix@-0.3ex{
\overline{U_{i}}^{}\ar^{\psi}[r]\ar@{>->}[d]  & U_{i}^{}\ar@{>->}[d]\\
\overline{\Proj{\C}}\ar^{\Psi}[r] &  \Proj{\C}
}
$$
with $\psi$ defined in (\ref{iso de abiertos}). As before, for every $A\in Comm(\C)$, we identify the object $(L,\tilde{\bf x})$ in $\overline{U_{i}}(A)$ with $(A,{\bf x})$. Take $(A,{\bf x})$ in $\overline{U_{i}}(A)$, then $\Psi(A,{\bf x})=(A,{\bf x}^{\vee})$, since ${\bf x}\lambda_{i}$ is an isomorphism and for all $j=0,\dots n$, $\lambda_{j}^{\vee}=\pi_{j}$, then $\lambda_{i}^{\vee}{\bf x}^{\vee}=\pi_{i}{\bf x}^{\vee}$ is an isomorphism. This says that the pair $(A,{\bf x}^{\vee})$ is in $U_{i}(A)$.  On the other hand, $\psi(A,{\bf x})=(A,{\bf y})$ with ${\bf y}: A\to A^{n+1}$ satisfying that $\pi_{j}{\bf y}={\bf x}\lambda_{j}$. To prove the commutativity of the diagram, that is, the compatibility between $\Psi$ and $\psi$, it is enough to show that both pairs $(A,{\bf x}^{\vee})$ and $(A,{\bf y})$ are the same subobject in $A^{n+1}$. The result follows from the fact that the dual of the morphism ${\bf x}\lambda_{j}:A\to A$ is itself in $\mod_{\C}(A)$, therefore:
$$
\pi_{j}{\bf x}^{\vee}=\lambda_{j}^{\vee}{\bf x}^{\vee}=({\bf x}\lambda_{j})^{\vee}={\bf x}\lambda_{j}=\pi_{j}{\bf y}
$$
for every $j=0,\dots n$, then ${\bf x}^{\vee}={\bf y}$.\\
To finish the proof, consider the commutative diagram in $Sh(Aff_{\C})$
$$
\xymatrix@-0.3ex{
\coprod_{i}\overline{U_{i}}^{}\ar^{\psi}[r]\ar_{\cong}[r]\ar@{->>}[d]  & \coprod_{i} U_{i}^{}\ar@{->>}[d]\\
\overline{\Proj{\C}}\ar^{\Psi}[r] &  \Proj{\C}
}
$$
therefore $\Psi$ is a sheaf epimorphism.
\end{proof}

\begin{proposition}\label{intersection}
The fiber product $U_{ij}=U_{i}\times_{\Proj{\C}}U_{j}$ is representable by an affine scheme.
\end{proposition}
\begin{proof}
For any  $A$ in $Comm({\C})$, an element in $U_{ij}(A)$ is an isomorphism class of pairs $(L,\bf x)$ where $\xymatrix@-1.2pc{L\ar@{^{(}->}[r]^{\bf x} & A^{n+1}}$ satisfies that $\pi_{i}{\bf { x}},\ \ \pi_{j}{ \bf{ x}}: L\to A$ are isomorphisms. We denote these isomorphisms by $x_{i},~ x_{j}$ respectively. Using these isomorphisms, we identify the pair $(L,\bf x)$ with a family of arrows 
$$
\left(\frac{x_{k}}{x_{i}}:A\to A\right)\quad\text{for}\quad k=0,\dots\hat i,\dots n,
$$
with the property that $\frac{x_{j}}{x_{i}}$ is an isomorphism (thence invertible). By the universal property of the localization and the polynomial algebra,  we have that 
$$U_{ij}(A)\cong Hom_{Comm(\C)}(\1[\tfrac{x_{0}}{x_{i}},\dots, \tfrac{x_{n}}{x_{i}}][\tfrac{x_{j}}{x_{i}}]^{\tiny {-1}},A)
$$
\end{proof}

%\begin{proposition}
%Let $A$ be a unital commutative ring and $\C=\mod_{\ent}(A)$, then $\Proj{\C}=\Proj{A}=\zp{n}\times_{\ent} A$???. Probarlo....  
%\end{proposition}
\subsection{The Octonionic Projective Space}
%\subsubsection*{The category of $\oct$-modules}
In \cite{quasialgebraO} the authors considered the symmetric monoidal category of real $G$-graded vector spaces $\mathcal U=(Vect^{G}_{\re},\X{G},\re,\Phi_{F},\sigma_{F})$ with $G=\ent_{2}\times\ent_{2}\times\ent_{2}$ and $F(x,y)=(-1)^{f(x,y)}$ with 
\begin{align*}
f(x,y)=&\sum_{i\leq j}x_{i}y_{j}+y_{1}x_{2}x_{3}+x_{1}y_{2}x_{3}+x_{1}y_{2}y_{3},\\
\phi_{F}(x,y,z)=&\partial F=\displaystyle\frac{F(x,y)F(xy,z)}{F(y,z)F(x,yz)},\\
\Phi_{F}((x\x y)\x z)=&\phi_{F}(x,y,z)x\x(y\x z)\quad\text{associativity constraint}\\
\sigma_{F}(x,y)=&\frac{F(x,y)}{F(y,x)}y\x x\qquad\qquad \text{symmetry}
\end{align*}
where by abuse of notation the degree of an homogeneous element is denoted by $|x|=x$. Then they proved that the Cayley algebra of the octonions $\oct$ can be obtained as the commutative algebra  $(\re\ent_{2}^{3},m_{F},\eta)$ in $\mathcal U$, with multiplication and unity given by $m_{F}(x,y)=F(x,y)x\cdot_{G} y$, $~\eta(1)=1_{G}$. Once we have a commutative algebra in a symmetric monoidal category, we can construct its category of modules and make some other constructions similar to those, one has in commutative algebra with the purpose to imitate the algebraic geometry over commutative rings.\\
In this section we will work on the properties of the category $\mod_{\mathcal U}(\oct)$, concerning to projective and free objects (respect to a left adjoint functor called the free functor). We will prove that in fact $\oct$ is a projective, finitely presented generator for the category $\mod_{\mathcal U}(\oct)$, this will say by using Gabriel's Theorem, that $\mod_{\mathcal U}(\oct)$ is in fact equivalent to a category of modules over a certain ring. A proof of this result can be seen in \cite{panaite2004quasi}. Although we are not interested in using this equivalence, it is worth to mention it.\\

Let us start by characterizing the objects in $\mod_{\mathcal U}(\oct)$. They consist of a pair $(X,\rho)$ with $X$ a $\ent_{2}^{3}$-graded real vector space with a graded morphism $\rho:\oct\X{\mathcal U} X\to X$ satisfying the pentagon and triangle axioms for the action. Since $\rho$ is a degree preserving morphism, then to give $\rho$ is to give an $8$-tuple of real vector space morphisms $\rho_{i}:(\oct\X{\mathcal U} X)_{i}\to X_{i}$, where the index denotes the $i$-th degree component. \\
If we denote $\{e_{i},\ \ i=0\dots 7\}$ a basis for $\oct$, then the associativity of the action says that $\rho(e_{i},\rho(e_{i},x_{k}))=-x_{k}$, this means that for every $i=0,\dots 7$, the multiplication by $e_{i}$ induces an isomorphism $X_{k}\cong X_{l}$ with $k, l$ such that $e_{l}=m_{F}(e_{i},e_{k})$.\\ In summary, an $\oct$-module is just a graded vector space with distinguished isomorphisms between the homogenous components, given by the multiplication of the basis elements of $\oct$. Thus, the data of being an $\oct$-module is in the $0$-th degree component and one obtains the rest of the components by multiplication of the $e_{i}'s$.\\
Next, a morphism between objects in $\mod_{\mathcal U}(\oct)$ will be a preserving degree morphism between the graded vector spaces compatible with the actions of $\oct$, i.e., a morphism between the degree zero components commuting with the respective isomorphisms. More explicitely, if $X,Y$ are objects in $\mod_{\mathcal U}(\oct)$, then $f:X\to Y$ is characterized by the morphism $f_{0}:X_{0}\to Y_{0}$, since the rest of the morphisms are just conjugations of $f_{0}$ by the $e_{i}$'s as is depicted in the following commutative diagram:
\begin{equation}\label{O-morphisms}
\xymatrix{
X_{0}\ar^{f_{0}}[r]\ar_{\cdot e_{i}}[d]  & Y_{0}\ar^{\cdot e_{i}}[d]\\
X_{i}\ar_{f_{i}}[r]  	&   Y_{i}.
}
\end{equation}

  All this implies the following propositions:

\begin{proposition}\label{V0 conservative}
Let $V_{0}=Hom_{\oct}(\oct,-):\mod_{\mathcal U}(\oct)\to \mathcal E ns$ be the ''canonical'' forgetful functor in monoidal categories. Then $V_{0}$ is a conservative functor.
\end{proposition}
%\begin{obs}
%The notation for this forgetful functor comes from the enriched category setting, we use it here although we are not constructing the category of modules as an enriched category over $\U$. In the Appendix, however we work in the enriched context.
%\end{obs}
\begin{proof}
Let $f:(X,\rho)\to (Y,\rho')$ be a morphism in $\mod_{\mathcal U}(\oct)$, such that the induced morphism $V_{0}(f):Hom_{\oct}(\oct,X)\to Hom_{\oct}(\oct,Y)$ is an isomorphism. If we denote by $|\text{-}|$ the forgetful functor, the adjunction 
$$
\xymatrix{\mod_{\mathcal U}(\oct)\ar@<-1ex>_{\quad\ \ |-| }^{\quad\ \ \perp}[r]& \U\ar@<-1ex>_{\ \ \ \ \ \ -\tens{\U}\oct}[l]}
$$
says that we have an isomorphism $Hom_{\U}(\re,|X|)\overset{\sim}{\to}Hom_{\U}(\re,|Y|)$. Since in the category $\mathcal U$ we have the isomorphisms $X_{0}\cong Hom_{\U}(\re,|X|)$, $Y_{0}\cong Hom_{\U}(\re,|Y|)$, hence we have the isomorphism $f_{0}:X_{0}\to Y_{0}$.  Finally, by the diagram \ref{O-morphisms}, we have that $f:(X,\rho)\to (Y,\rho')$ is in fact an isomorphism.
\end{proof}

\begin{proposition}\label{O is projective}
$\oct$ is a projective finitely presented generator in $\mod_{\mathcal U}(\oct)$.
\end{proposition}

\begin{proof}
Limits and colimits in $\mod_{\mathcal U}(\oct)$ are computed in $\mathcal U$, this means in particular that $|\text{-}|:\mod_{\mathcal U}(\oct)\to \mathcal U$ preserves them. Now, since $\re$ is a projective object in $\mathcal U$, then $\oct\cong \re\X{\U}\oct$ is projective in $\mod_{\mathcal U}(\oct)$.\\
Now, to show that $\oct$ is finitely presented, observe thar $Hom_{\mathcal U}(\re,-)$ preserves them, hence by the isomorphism 
$$
Hom_{\oct}(\oct,-)\cong Hom_{\mathcal U}(\re,|-|)
$$
we get that $Hom_{\oct}(\oct,-)$ preserves filtered colimits, that is $\oct$ is finitely presented.\\
Finally, to see that $\oct$ is a generator, we have to prove that $Hom_{\oct}(\oct,-)$ is faithful. The result follows by ``abstract nonsense'': In any category $\C$ with equalizers , a conservative functor $F:\C\to \mathcal Ens$ preserving them is faithful.
\end{proof}

From all the previous construction we obtain that $\mod_{\mathcal U}(\oct)$ is an abelian strong relative context, then we have the following result

\begin{cor}\label{projectivo octonionico}
We define the octonionic projective space $\Proj{\oct}$ to be the functor $\Proj{\C}$ relative to the category $\C=\mod_{\mathcal U}(\oct)$. $\Proj{\oct}$ is a relative scheme.
\end{cor}

%\bibliographystyle{unsrt}
%\bibliography{paperRelativeProjective}

\end{document}